\documentclass[11pt,twoside]{article}
\usepackage{amssymb,mathrsfs,color}
\usepackage{amsfonts,amsmath,latexsym}
\usepackage{enumerate}

\usepackage[amsmath,thmmarks]{ntheorem}
\usepackage{indentfirst}

\numberwithin{equation}{section}

\theoremheaderfont{\normalfont\rmfamily\bf}
\theorembodyfont{\normalfont\it }
\newtheorem{theorem}{\indent Theorem}[section]
\newtheorem{lemma}{\indent Lemma} [section]

\newtheorem{definition}{\indent Definition} [section]
\newtheorem{remark}{\indent Remark} [section]

  \theoremstyle{nonumberplain}

    \newenvironment{proof}[1][Proof]{\textbf{#1}~~}{\hfill $\square$ \bigskip}

\allowdisplaybreaks

\DeclareMathOperator*{\esssup}{ess\, sup}

\def\rn{{\mathbb{R}}^n}

\def\loc{{\mathrm{loc}}}

\begin{document}

\pagestyle{myheadings} \markboth{\small\it Pu Zhang and Di Fan}
 {\small\it Commutators for the maximal and sharp functions on weighted Morrey spaces}

\title{\bf{Commutators for the maximal and sharp functions
   with weighted Lipschitz functions  on weighted Morrey spaces }
  \footnotetext{This work was supported by the National
     Natural Science Foundation of China (Grant No. 11571160),
      the Reform and Development Foundation for Local Universities from Central Government of China (No. 2020YQ07).
        and the Scientific Research Fund of Mudanjiang Normal University (No. D211220637). }
    \footnotetext{E-mail:  puzhang@sohu.com (P. Zhang); fandi2024@163.com (D. Fan)}
        }

\author{Pu Zhang\footnote{Corresponding author: Pu Zhang}~ and~ Di Fan\\
           {\small\it Department of Mathematics, Mudanjiang Normal
                 University, Mudanjiang 157011, P. R. China}
   }

 \date{}
\maketitle

\begin{center}\begin{minipage}{14cm}
{\bf Abstract} ~
We study the boundedness of commutators of the Hardy-Littlewood
maximal function and the sharp maximal function on weighted Morrey spaces
 when the  symbols of the commutators belong to weighted Lipschitz spaces.
  Some new characterizations for weighted Lipschitz functions are obtained
   in term of the boundedness of the commutators.


\end{minipage}
\end{center}

\maketitle

\section{Introduction and Main Results}

Let $T$ be the  classical singular integral operator. The commutator
$[b,T]$ generated by $T$ and a suitable function $b$ is given by
\begin{equation} \label{equ:1.1}    
[b,T](f)(x) = T\big((b(x)-b)f\big)(x)
       = b(x)T(f)(x)-T(bf)(x).
\end{equation}

A well known result states that $[b,T]$ is bounded on $L^p(\rn)$
for $1<p<\infty$ if and only if $b\in {BMO(\rn)}$
(see Coifman, Rochberg and Weiss \cite{CRW} and Janson \cite{J1978}).
Another kind of boundedness of $[b,T]$ was given
by Janson \cite{J1978} in 1978, see also Paluszy\'nski \cite{P1995}.
It was proved that $[b,T]$ is bounded from $L^p(\rn)$ to $L^q(\rn)$
for $1<p<q<\infty$ if and only if $b\in {Lip_{\beta}(\rn)}$ with
$0<\beta=n(1/p-1/q)<1$,
where ${Lip_{\beta}(\rn)}$ is the Lipschitz space of order $\beta$.
In 2008, Hu and Gu \cite{HG} considered the  weighted boundedness
of commutator $[b,T]$ when $b$ belongs to weighted Lipschitz spaces.
In 2009, Komori and Shirai \cite{KS} introduced weighted Morrey spaces
 and studied the boundedness of some classical operators
  in harmonic analysis on such spaces,
   including the commutators of singular integral.
In 2012, Wang \cite{W2012} considered some boundedness properties
 of commutator $[b,T]$ on the weighted Morrey spaces
  when the symbol $b$ belongs to weighted BMO or weighted Lipschitz spaces.

On the other hand, the boundedness properties of commutators of
 the Hardy-Littlewood maximal function and sharp maximal function
  have been studied intensively by many authors in recent years.
   See \cite{AGKM, BMR, FJ2012, GD, MS, Xie, Z2017, ZW2014}, for instance.

Let $Q$ be a cube in $\rn$ with sides
parallel to the coordinate axes.
 We denote by $|Q|$ the Lebesgue measure and by $\chi_Q$
the characteristic function of $Q$.
Let $f$ be a locally integrable function on $\rn$.
The Hardy-Littlewood maximal function of $f$ is defined by
$$M(f)(x)=\sup_{Q\ni x} \frac{1}{|Q|} \int_{Q} |f(y)| dy,
$$
and the sharp maximal function $M^{\sharp}$  
  is given by
$$M^{\sharp}(f)(x)=\sup_{Q\ni x}\frac{1}{|Q|} \int_{Q}|f(y)-f_{Q}|{d}y,
$$
where the supremum is taken over all cubes $Q\subset \rn$
containing $x$ and $f_Q={|Q|}^{-1}\int_Q f(x)dx$.

Similar to (\ref{equ:1.1}), we can define two different kinds of
commutators of the Hardy-Littlewood maximal function as follows.

Let $b$ be a locally integrable function. The maximal commutator generated
by $M$ and $b$ is given by
$$M_b(f)(x)= M\big((b(x)-b)f\big)(x)
=\sup_{Q\ni x} \frac{1}{|Q|} \int_{Q} |b(x)-b(y)||f(y)|dy,
$$
where the supremum is taken over all cubes $Q\subset\rn$ containing
$x$.

The (nonlinear) commutator generated by $M$ and $b$ is defined by
$$[b,M](f)(x)=b(x)M(f)(x)-M(bf)(x).
$$

Similarly, we can also define the commutator generated by
$M^{\sharp}$ and $b$ by
$$[b,M^{\sharp}](f)(x)= bM^{\sharp}(f)(x) -M^{\sharp}(bf)(x).
$$

Obviously, commutators $M_b$ and $[b,M]$ essentially
 differ from each other. 
 $M_b$ is positive and sublinear, but $[b,M]$ and $[b,M^{\sharp}]$
  are neither positive nor sublinear.
 The operator $[b,M]$ can be used in studying the product
  of a function in $H^1$ and a
    function in $BMO$, see \cite{BIJZ} for example.

In 1990, Milman and Schonbek \cite{MS} showed that $[b,M]$
is bounded on $L^p(\rn)$ $(1<p<\infty)$ when $b\in BMO(\rn)$ and $b\geq 0$.
 In 2000, Bastero, Milman and Ruiz \cite{BMR} gave some
 characterizations for the boundedness of $[b,M]$ and
 $[b,M^{\sharp}]$ on $L^p$ spaces. Certain BMO classes are characterized
 by the boundedness of the commutators.
  Xie \cite{Xie} extended the results to the context of Morrey spaces.
 Recently, Zhang in \cite{Z2017} and \cite{Z2019} considered
  the mapping properties of $M_{b}$, $[b,M]$ and $[b,M^{\sharp}]$
   when the symbols $b$ belong to Lipschitz space.
 Some necessary and sufficient conditions for the boundedness of
 the commutators on Lebesgue and Morrey spaces are given,
   by which some new characterizations of Lipschitz functions are obtained.
 The results were extended to variable exponent
  Lebesgue spaces in \cite{Z2019}, \cite{ZSW} and \cite{YYL},
     and to the context of Orlicz spaces in \cite{Gul}, \cite{GD}, \cite{GDH} and \cite{ZWS}.

Most recently, the boundedness of 
 $M_{b}$, $[b,M]$ and $[b,M^{\sharp}]$ on weighted
Lebesgue spaces are characterized when the symbols $b$ belong to
weighted Lipschitz spaces by Zhang and Zhu in \cite{ZZ}.
 Some new characterizations for weighted Lipschitz functions are also given.

Motivated by the results mentioned above, we will
 consider the mapping properties of the commutators on
  weighted Morrey spaces when the symbols $b$ belong to weighted Lipschitz spaces.

To state our results, let us recall some definitions and notations.

\begin{definition}[\cite{GarR}, \cite{St}]  \label{def:1.1}
A weight is a nonnegative locally integrable function on $\rn$
that takes values in $(0,\infty)$ almost everywhere.
 Let $\mu$ be a weight.

 (1)  We say that $\mu \in A_p$ for $1<p<\infty$,
 if there exists a constant $C>0$ such that for any cube $Q$ in $\rn$,
$$\bigg(\frac{1}{|Q|}\int_{Q}\mu(x)dx\bigg)
  \bigg(\frac{1}{|Q|}\int_{Q}\mu(x)^{1-p'}dx\bigg)^{p-1} \leq C,
$$
here and below, $p'$ denotes the conjugate exponent of $p$,
that is $1/p+1/p'=1$.

 (2)  We say that $\mu \in A_1$ if there is a constant $C>0$
such that for any cube $Q\subset \rn$,
$$\frac{1}{|Q|}\int_{Q}\mu(y)dy \leq C\mu(x)~~ \hbox{a.e.}~ x \in Q.
  $$

 (3) We define $A_{\infty}=\bigcup\limits_{1\le{p}<\infty}A_{p}.$
\end{definition}

Let $\mu$ be a weight. For a cube $Q$ and a function $f$, 
we write $\mu(Q)=\int_{Q}\mu(x)dx$ and
$$\|f\|_{L^{p}(\mu)}=\bigg(\int_{\mathbb{R}^{n}}|f(x)|^{p}\mu(x)dx\bigg)^{1/p}.
$$

Following \cite{Gar}, we define the weighted Lipschitz function spaces.
See also \cite{HG} and \cite{T2011}.

\begin{definition}     \label{def:1.2}     
Let $1\le{p}\le\infty$, $0<\beta<1$ and $\mu\in{A_{\infty}}$.
The weighted Lipschitz space, denoted by $Lip_{\beta,\mu}^p$, is given by
$$Lip_{\beta,\mu}^p
  =\Big\{f \in{L_{\rm{loc}}^1(\rn)}: \|f\|_{Lip_{\beta ,\mu}^p}<\infty \Big\},
$$
where
$$\|f\|_{Lip_{\beta ,\mu}^p} = \sup_Q \frac{1}{\mu(Q)^{\beta /n}}
 {\bigg(\frac{1}{\mu(Q)} \int_Q |f(x)-f_Q|^p}\mu (x)^{1-p}dx\bigg)^{1/p},
   ~~\mbox{for}~ 1\le p<\infty,
$$
and
$$\|f\|_{Lip_{\beta ,\mu}^{\infty}}
   = \sup_Q \frac{1}{\mu(Q)^{\beta/n}} \esssup_{x\in{Q}} \frac{|f(x)-f_Q|}{\mu(x)},
$$
here and below,
 ``$\sup\limits_Q$'' always means the supremum is taken over all
 cubes $Q\subset \rn$.
\end{definition}

Modulo constants, $Lip_{\beta,\mu}^p$ is a Banach space
with respect to the norm $\|\cdot\|_{Lip_{\beta ,\mu}^p}$.
We simply write $Lip_{\beta,\mu} =Lip_{\beta ,\mu}^1$.
 Obviously the space $Lip_{\beta,\mu}^p$ is
a special case of the so-called weighted Morrey-Campanato spaces
studied by Garc\'ia-Cuerva \cite{Gar} and other authors, see for example
\cite{T2011}, \cite{YY} and \cite{SDH}, and references therein.

 The study of Morrey spaces goes back to the work of Morrey \cite{M1938},
which investigated the local behavior of solutions to second order
 elliptic partial differential equations.
In 2009, Komori and Shirai \cite{KS} introduced the
 weighted Morrey spaces as follows.

\begin{definition}[\cite{KS}]         
Let $1\leq p<\infty$ and $0<\kappa<1$. For two weights $u$ and $v$,
the weighted Morrey space $L^{p,\kappa}(u,v)$ is defined by
\begin{align*}
L^{p,\kappa}(u,v)=\big\{f\in L_{\loc}^{p}(u):
     \|f\|_{L^{p,\kappa}(u,v)}<\infty\big\},
\end{align*}
where
\begin{align*}
\|f\|_{L^{p,k}(u,v)}
 =\sup_{Q}\bigg(\frac{1}{v(Q)^{\kappa}}\int_{Q}|f(x)|^{p}u(x)dx\bigg)^{1/p}.
\end{align*}

If $u=v=\mu$, we write $L^{p,\kappa}(u,v)=L^{p,\kappa}(\mu),$
 the weighted Morrey space with one weight.
\end{definition}

Given a cube $Q_0$, we need the following locally maximal function
 with respect to $Q_0$,
$$M_{Q_0}(f)(x) =\sup_{Q_0\supseteq{Q}\ni{x}}
\frac{1}{|Q|}\int_{Q}|f(y)|dy,
$$
where the supremum is taken over all cubes $Q\subseteq{Q_0}$
 and $x\in{Q}$.

We are now in a position to state our results.

\begin{theorem}   \label{thm.1.1}  
Let $b$ be a locally integrable function, $\mu\in {A_1}$ and $0<\beta<1$.
Suppose that $1<p<n/\beta$, $1/q=1/p-\beta/n$ and $0<\kappa <p/q$.
Then the following assertions are equivalent:

 {\rm (1)}  $b\in Lip_{\beta,\mu}$.

 {\rm (2)}  $M_{b}$ is bounded from $L^{p,\kappa}(\mu)$
    to $L^{q,{\kappa}q/p}(\mu^{1-q},\mu)$.
\end{theorem}

\begin{theorem}      \label{thm.1.2}  
Let $b$ be a locally integrable function, $\mu\in {A_1}$ and $0<\beta<1$.
Suppose that $1<p<n/\beta$, $1/q=1/p-\beta/n$ and $0<\kappa <p/q$.
Then the following assertions are equivalent:

 {\rm (1)}  $b\in Lip_{\beta,\mu}$ and $b\geq 0$ a.e. in $\rn$.

 {\rm (2)}  $[b,M]$ is bounded from $L^{p,\kappa}(\mu)$
    to $L^{q,{\kappa}q/p}(\mu^{1-q},\mu)$.

 {\rm (3)}  There exists a constant $C>0$ such that
\begin{equation}   \label{equ:1.2}
\sup_{Q}\dfrac{1}{\mu(Q)^{\beta/n}}\bigg(\frac{1}{\mu(Q)}
  \int_{Q}\big|b(x)-M_{Q}(b)(x)\big|^{q}\mu(x)^{1-q}dx\bigg)^{1/q}\leq C.
\end{equation}
\end{theorem}

\begin{remark}
 When $\mu\equiv 1$ the results of Teorems \ref{thm.1.1}
    and \ref{thm.1.2} were proved in \cite{Z2017}.
\end{remark}

\begin{theorem}   \label{thm.1.3}  
Let $b$ be a locally integrable function, $\mu\in {A_1}$
 and $0<\beta<1$.
Suppose that $1<p<n/\beta$, $1/q=1/p-\beta/n$ and $0<\kappa <p/q$.
Then the following assertions are equivalent:

 {\rm (1)} $b\in Lip_{\beta,\mu}$ and $b\geq 0$ a.e. in $\rn$.

 {\rm (2)} $[b,M^{\sharp}]$ is bounded from $L^{p,\kappa}(\mu)$
    to $L^{q,{\kappa}q/p}(\mu^{1-q},\mu)$.

 {\rm (3)}   There exists a constant $C>0$ such that
\begin{align}\label{equ:1.3}
\begin{split}
\sup_{Q}\frac{1}{\mu(Q)^{\beta/n}}
\bigg(\frac{1}{\mu(Q)}\int_{Q}\big|b(x)
 -2M^{\sharp}(b\chi_{Q})(x)\big|^{q}\mu(x)^{1-q}dx\bigg)^{1/q}\leq C.
\end{split}
\end{align}
\end{theorem}

\begin{remark}
Theorem \ref{thm.1.3} is new even for the case $\mu\equiv 1$.
 The authors studied the boundedness for $[b,M^{\sharp}]$
 on Lebesgue spaces when $b$ belongs to Lipschitz spaces
  in \cite{ZWS} and then extended the result to variable Lebesgue spaces in \cite{Z2019}.
   When $b$ belongs to weighted Lipschitz spaces
   the boundedness of $[b,M^{\sharp}]$ on weighted Lebesgue spaces are obtained
   by Zhang and Zhu in \cite{ZZ}.
\end{remark}

\section{Preliminaries and Lemmas}

In this section, we present some lemmas which will be used 
 in the proof of our results.
In 2011, Tang \cite{T2011} studied the weighted Morrey-Campanato spaces
and established some new characterizations for such spaces.
 We rewrite his Theorem 2.1 as follows to adapt to our cases of
weighted Lipschitz spaces.
 For $x_0\in\rn$ and $r>0$, denote by $B(x_0,r)$
   the ball centered at $x_0$ with radius $r$.

\begin{lemma}[\cite{T2011}]    \label{lem.T2001}  
Let $0<\beta<1$, $w\in{A_1}$ and $1\le{p}\le\infty$.
Then $Lip_{\beta,\mu}^p =Lip_{\beta ,\mu}$ with equivalent norms.
Furthermore, $b\in {Lip_{\beta,w}^p}$
if and only if there exists a constant $C$ such that
$$|b(x)-b(y)|\le {C}\|b\|_{Lip_{\beta,w}}\big[w(B(x,|x-y|))\big]^{\beta/n}(w(x)+w(y))
$$
for a.e. $x, y \in \rn$, especially for $x,y$ being Lebesgue points of $b$.
\end{lemma}

By Lemma \ref{lem.T2001}, the following result is a consequence of
Definition \ref{def:1.2} for the case $p=\infty$. See also \cite{ZZ}.

\begin{lemma}  \label{lem.ZZ.b-bQ}  
Let $w\in{A_1}$, $0<\beta<1$ and $b\in {Lip_{\beta,w}}$.
 Then for any cube $Q\subset\rn$, there is a constant $C$, such that
\begin{equation}  \label{equ.lem.ZZ.b-bQ}  
|b(x)-b_Q|\le {C}\|b\|_{Lip_{\beta,w}} w(Q)^{\beta/n}w(x)
  ~ \mbox{for~a.e.} ~x\in {Q}.
\end{equation}
\end{lemma}

\begin{lemma}[\cite{KS}] \label{lem.KS}  
If $1<p<\infty$, $0<\kappa<1$ and $\mu\in A_{p}$,
 then the Hardy-Littlewood maximal operator $M$ is bounded on $L^{p,\kappa}(\mu)$.
\end{lemma}

\begin{lemma}   \label{lem.Q-Morrey}   
If $1<p<\infty$, $0<\kappa<1$ and $\mu\in A_{\infty}$,
 then, for any cube $Q$, we have
$$\|\chi_{Q}\|_{L^{p,\kappa}(\mu)}\leq \mu(Q)^{(1-\kappa)/{p}}.$$
\end{lemma}

\begin{proof}
For any fixed cube $Q$, noting that $0<\kappa<1$, we have
\begin{align*}
\|\chi_{Q}\|_{L^{p,\kappa}(\mu)}
&=\sup_{Q'}\bigg(\frac{1}{\mu(Q')^{\kappa}}
    \int_{Q'}|\chi_{Q}(x)|^{p}\mu(x)dx\bigg)^{1/p}\\
&=\sup_{Q'}\bigg(\frac{1}{\mu(Q')^{\kappa}}
   \int_{Q'\cap Q}|\chi_{Q}(x)|^{p}\mu(x)dx\bigg)^{1/p}\\
& \le \sup_{Q'}\bigg(\frac{1}{\mu(Q'\cap Q)^{\kappa}}
   \int_{Q'\cap Q}|\chi_{Q}(x)|^{p}\mu(x)dx\bigg)^{1/p}\\
& \le \sup_{Q'}\mu(Q'\cap Q)^{(1-\kappa)/p}\\
&\leq  \mu(Q)^{(1-\kappa)/p}.
\end{align*}
This proves the required conclusion.
\end{proof}

We also need some estimates for a weighted fractional maximal function.

\begin{lemma}[\cite{W2012}]   \label{lem.W2012}  
Let $0<\beta<n, 1< p<n/\beta, 1/q = 1/p-\beta/n$ and $0<\kappa<p/q$.
Suppose that $\mu\in A_{\infty}$, then for any  $1<r<p$, we have
\begin{align*}
\|M_{\beta,\mu,r}(f)\|_{L^{q,\kappa{q}/p}(\mu)}\leq C\|f\|_{L^{p,\kappa}(\mu)},
\end{align*}
where $M_{\beta,\mu,r}$ is a weighted fractional maximal function given by
$$M_{\beta,\mu,r}(f)=\sup_{Q\ni x}
 \bigg(\frac{1}{\mu(Q)^{1-r\beta/n}}\int_{Q}|f(y)|^{r}\mu(y)dy\bigg)^{1/r}.$$
\end{lemma}

\begin{lemma}[\cite{LWM}]   \label{lem.LWM}  
 Let $\mu\in A_{1}$, $0<\beta<1$ and $b\in Lip_{\beta,\mu}$.
 Then there exists a constant $C>0$ such that

{\rm (1)} for any $1\leq r<\infty$ and any cube $Q\ni x$, we have
$$\frac{1}{|Q|}\int_{Q}|f(y)|dy\leq C\mu(Q)^{-\beta/n} M_{\beta,\mu,r}(f)(x);
$$

{\rm(2)}  for any $1<r<\infty$ and any cube $Q\ni x$, we have
$$\frac{1}{|Q|}\int_{Q}\big|(b(y)-b_{Q})f(y)\big|dy
\leq C\|b\|_{Lip_{\beta,\mu}}\mu(x)M_{\beta,\mu,r}(f)(x).
$$
\end{lemma}

Now, we show that the weighted fractional maximal function
 controls the maximal commutator $M_b$
 when the symbol $b$ belongs to weighted Lipschitz space.

\begin{lemma}   \label{lem.ZZ-Mb}      
Let $\mu\in A_{1}, 0<\beta<1$ and $b\in Lip_{\beta,\mu}.$
 Then for any locally integrable function $f$ and any $1<r<\infty$,
 we have
$$M_{b}(f)(x)\leq C\|b\|_{Lip_{\beta,\mu}}\mu(x)M_{\beta,\mu,r}(f)(x)
 ~~ \mbox{a.e.}~ x\in \rn.
 $$
\end{lemma}

\begin{proof} Given any $1<r<\infty$. For any $x$ satisfying (\ref{equ.lem.ZZ.b-bQ})
 and any cube $Q\ni{x}$, by Lemmas \ref{lem.ZZ.b-bQ} and \ref{lem.LWM} we have
\begin{align*}
&\frac{1}{|Q|}\int_{Q}\big|b(x)-b(y)\big|\big|f(y)\big|dy\\
&\leq \big|b(x)-b_{Q}\big|\frac{1}{|Q|}\int_{Q}|f(y)|dy
  +\frac{1}{|Q|}\int_{Q}\big|b(y)-b_{Q}\big|\big|f(y)\big|dy\\
&\leq C\|b\|_{Lip_{\beta,\mu}}\mu(x)M_{\beta,\mu,r}(f)(x).
\end{align*}
Therefore, we obtain
\begin{align*}
M_{b}(f)(x)\leq C\|b\|_{Lip_{\beta,\mu}}\mu(x)M_{\beta,\mu,r}(f)(x)
 ~~\mbox{for a.e.}~ x\in \rn,
\end{align*}
which completes the proof.
\end{proof}

To finish this section, we recall the following
 characterizations for weighted Lipschitz function obtained
  by Zhang and Zhu in \cite{ZZ}.

\begin{lemma}   \label{lem.ZZ.lip} 
Let $0<\beta<1$, $b$ be a locally integrable function
 and $\mu\in {A_1}$.
Then the following statements are equivalent:

{\rm(1)} $b \in Li{p_{\beta ,\mu }}$ and $b\ge 0$ a.e. in $\rn$.

{\rm(2)} There exists $1\le{s}<\infty$ such that
\begin{equation}   \label{equ.lem.ZZ.lip1}
\sup_B \frac{1}{\mu(Q)^{\beta/n}}
 \bigg(\frac{1}{\mu(Q)}\int_Q |b(x)-M_Q(b)(x)|^s \mu(x)^{1-s} dx \bigg)^{1/s} \le{C}.
\end{equation}

{\rm(3)} For all $1\le{s}<\infty$, (\ref{equ.lem.ZZ.lip1}) holds.

{\rm(4)} There exists $1\le{s}<\infty$ such that
\begin{equation}   \label{equ.lem.ZZ.lip2}
 \sup_Q \frac{1}{\mu(Q)^{\beta/n}}
  \bigg(\frac{1}{\mu(Q)}\int_Q |b(x)-2M^{\sharp} (b\chi_Q)(x)|^s \mu(x)^{1-s} dx \bigg)^{1/s} \le{C}.
\end{equation}

{\rm(5)} For all $1\le{s}<\infty$, (\ref{equ.lem.ZZ.lip2}) holds.
\end{lemma}

\section{Proof of Theorems}

\begin{proof}[Proof of Theorem \ref{thm.1.1}] 
We first proof the implication $(1)\Rightarrow (2)$.
  For any $r$ satisfying $1<r<p$, by Lemma \ref{lem.ZZ-Mb} we have
\begin{align*}
\|M_{b}(f)\|_{L^{q,\kappa{q}/p}(\mu^{1-q},\mu)}
  &= \sup_{Q}\bigg(\frac{1}{\mu(Q)^{\kappa{q/p}}}
      \int_{Q}\big[M_{b}(f)(x)\big]^{q}\mu(x)^{1-q}dx\bigg)^{1/q}\\
&\leq C\|b\|_{Lip_{\beta,\mu}} \sup_{Q}\bigg(\frac{1}{\mu(Q)^{\kappa{q/p}}}
     \int_{Q} \big[M_{\beta,\mu,r}(f)(x)\big]^{q}\mu(x)dx\bigg)^{1/q}\\
&\leq C\|b\|_{Lip_{\beta,\mu}}\|M_{\beta,\mu,r}(f)\|_{L^{q,\kappa{q/p}}(\mu)}\\
& \le C\|b\|_{Lip_{\beta,\mu}} \|f\|_{L^{p,\kappa}(\mu)},
\end{align*}
where the last step follows from Lemma \ref{lem.W2012}.
 This completes the proof of $(1)\Rightarrow (2)$.

Next we proof the implication $(2)\Rightarrow (1)$.
Suppose that assertion (2) holds, 
 it suffices to prove that there is a constant $C>0$
  such that, for all cubes $Q\subset \rn$, we have
\begin{equation}    \label{equ:thm.1.1-1}
\frac{1}{\mu(Q)^{1+\beta/n}}\int_Q\big|b(x)-b_{Q}\big|dx\le{C}.
\end{equation}

For any cube $Q\subset \rn$, applying H\"{o}lder's inequality,
 Lemma \ref{lem.Q-Morrey} and assertion (2), 
 and noting that$1/p=1/q+\beta/n$, we obtain
\begin{align*}
\frac{1}{\mu(Q)^{1+\beta/n}}\int_Q\big|b(x)-b_{Q}\big|dx
&\leq \frac{1}{\mu(Q)^{1+\beta/n}}\int_Q\bigg(\frac{1}{|Q|}\int_Q |b(x)-b(y)|\chi_{Q}(y)dy\bigg)dx\\
&\leq \frac{1}{\mu(Q)^{1+\beta/n}}\int_Q M_{b}(\chi_{Q})(x)dx\\
   &= \frac{1}{\mu(Q)^{1+\beta/n}}\int_Q M_{b}(\chi_{Q})(x)\mu(x)^{(1-q)/q} \mu(x)^{(q-1)/q}dx\\
&\leq \frac{1}{\mu(Q)^{1+\beta/n}}\bigg(\int_Q \big[M_{b}(\chi_{Q})(x)\big]^{q}\mu(x)^{1-q}dx\bigg)^{1/q}
            \bigg(\int_Q  \mu(x) dx\bigg)^{1-1/q}\\
&\leq \frac{1}{\mu(Q)^{1/p-{\kappa/p}}}
    \bigg( \frac{1}{\mu(Q)^{\kappa{q/p}}}
     \int_Q \big[M_{b}(\chi_{Q})(x)\big]^{q}\mu(x)^{1-q}dx\bigg)^{1/q}\\
&\leq \frac{1}{\mu(Q)^{(1-\kappa)/p}}
     \|M_{b}(\chi_{Q})\|_{L^{q,\kappa{q/p}}(\mu^{1-q},\mu)}\\
&\leq \frac{C}{\mu(Q)^{(1-\kappa)/p}} \|\chi_{Q}\|_{L^{p,\kappa}(\mu)}\\
 & \le {C},
\end{align*}
which is nothing other than (\ref{equ:thm.1.1-1}) since the constant $C$ is independent of $Q$.

The proof of Theorem \ref{thm.1.1} is complete.
\end{proof}

\begin{proof}[Proof of Theorem \ref{thm.1.2}] 
Observe that the equivalence of (1) and (3) follows readily from Lemma \ref{lem.ZZ.lip}.
 We only need to check the implications $(1)\Rightarrow(2)$ and $(2)\Rightarrow(3)$.

$(1)\Rightarrow (2)$. ~Given $x\in\rn$ such that
 $M(f)(x)<\infty$ and $0\le {b(x)}<\infty$.
  Noting that $b\geq 0$ a.e. in $\rn$, we have
\begin{equation}  \label{equ:thm.1.2-1}  
\begin{split}
\big|[b,M](f)(x)\big| &=\big|b(x)M(f)(x)-M(bf)(x)\big|\\
&=\bigg|\sup_{Q\ni x}\frac{1}{|Q|}\int_{Q}b(x)|f(y)|dy
    -\sup_{Q\ni x}\frac{1}{|Q|}\int_{Q}b(y)|f(y)|dy\bigg|\\
&\leq \mathop{\sup}\limits_{Q\ni x}\frac{1}{|Q|}\int_{Q}|b(x)-b(y)||f(y)|dy\\
&= M_{b}(f)(x).
\end{split}
\end{equation}

Observe that if $f\in L^{p,\kappa}(\mu)$ then, by Lemma \ref{lem.KS},
we have $M(f)\in L^{p,\kappa}(\mu)$,
which implies that $M(f)<\infty$ a.e. in $\rn$,
 and note that $b$ is finite a.e. in $\rn$ since $b$ is locally integrable.
  Then, we conclude that (\ref{equ:thm.1.2-1}) holds for a.e. $x\in\rn$.

  Thus, by Theorem \ref{thm.1.1} we obtain that $[b,M]$ is
 bounded from $L^{p,\kappa}(\mu)$ to $L^{q,\kappa{q/p}}(\mu^{1-q},\mu)$.

$(2)\Rightarrow (3)$.
 For any fixed $Q \subset\rn$ and $x \in{Q}$,
we have (see \cite{BMR} page 3331 or (2.4) in \cite{ZW2009})
$$M (\chi_Q)(x) = {\chi_Q}(x) \hbox{~~and~~} M (b \chi_Q)(x) = {M_Q}(b)(x).
$$

Noting that $1/p=1/q+\beta/n$ and applying assertion (2)
  and Lemma \ref{lem.Q-Morrey}, we have
\begin{align*}
&\frac{1}{\mu(Q)^{\beta/n}}\bigg(\frac{1}{\mu(Q)}
  \int_{Q}\big|b(x)-M_{Q}(b)(x)\big|^{q}\mu(x)^{1-q}dx\bigg)^{1/q}\\
&=\frac{1}{\mu(Q)^{\beta/n}}\bigg(\frac{1}{\mu(Q)}
  \int_{Q}\big|b(x)M(\chi_{Q})(x)-M_{Q}(b\chi_Q)(x)\big|^{q}\mu(x)^{1-q}dx\bigg)^{1/q}\\
&=\frac{\mu(Q)^{\kappa/p}}{\mu(Q)^{\beta/n+1/q}}
   \bigg(\frac{1}{\mu(Q)^{\kappa{q/p}}}
     \int_{Q}\big|[b,M](\chi_{Q})(x)\big|^{q}\mu(x)^{1-q}dx\bigg)^{1/q}\\
&\leq {\mu(Q)^{(\kappa-1)/p}}
      \big\|[b,M](\chi_{Q})\big\|_{L^{q,\kappa{q/p}}(\mu^{1-q},\mu)}\\
&\leq C{\mu(Q)^{(\kappa-1)/p}} \|\chi_{Q}\|_{L^{p,\kappa}(\mu)}\\
&\leq C.
\end{align*}
Since the constant $C$ is independent of $Q$, then we conclude assertion (3).

So the proof of Theorem \ref{thm.1.2} is finished.
\end{proof}

\begin{proof}[Proof of Theorem \ref{thm.1.3}] 
Similar to the proof of Theorem \ref{thm.1.2},
 by Lemma \ref{lem.ZZ.lip}, we need to prove the implications
  $(1)\Rightarrow(2)$ and $(2)\Rightarrow(3)$.

$(1)\Rightarrow (2)$.
 For any $x\in\rn$ such that $M^{\sharp}(f)(x)<\infty$ and $0\le {b(x)}<\infty$, 
 noting that $b\geq 0$ a.e. in $\rn$, we have
\begin{equation}\label{equ:thm.1.3-1}
\begin{split}
\big|[b,M^{\sharp}](f)(x)\big|
&=\bigg|\sup_{Q\ni x}\frac{b(x)}{|Q|} \int_{Q}\big|f(y)-f_{Q}\big|{d}y
    -\sup_{Q\ni x}\frac{1}{|Q|}\int_{Q}\big|b(y)f(y)-(bf)_{Q}\big|{d}y\bigg|\\
 &\le \mathop{\sup}\limits_{Q\ni x}\frac{1}{|Q|}
  \bigg|\int_{Q}\big|b(x)f(y)-b(x)f_{Q}\big|{d}y
    -\int_{Q}\big|b(y)f(y)-(bf)_{Q}\big|{d}y\bigg|\\
&\leq \sup_{Q\ni x}\frac{1}{|Q|}\int_{Q}\big|\big(b(x)-b(y)\big)f(y)
     +b(x)f_{Q}-(bf)_{Q}\big|{d}y\\
&\leq \sup_{Q\ni x}\bigg\{\frac{1}{|Q|}\int_{Q} \big|b(x)-b(y)\big|\big|f(y)\big|{d}y
     +\big|b(x)f_{Q}-(bf)_{Q}\big|\bigg\}\\
&\leq M_{b}(f)(x)+\mathop{\sup}\limits_{Q\ni x}\frac{1}{|Q|}
    \int_{Q}\big|b(x)-b(z)\big|\big|f(z)\big| {d}z\\
&\leq 2M_{b}(f)(x).
\end{split}
\end{equation}

For $f\in L^{p,\kappa}(\mu)$, by Lemma \ref{lem.KS} again,
we have $M(f)<\infty$ a.e. in $\rn$.
 Then $M^{\sharp}(f)<\infty$ a.e. in $\rn$.
 Note that $b$ is finite a.e. in $\rn$,
  we obtain that (\ref{equ:thm.1.3-1}) holds  for a.e. $x\in \rn$.
 Therefore, it follows from Theorem \ref{thm.1.1} that
 $[b,M^{\sharp}]$ is bounded from $L^{p,\kappa}(\mu)$ to
  $L^{q,\kappa{q/p}}(\mu^{1-q},\mu)$.

$(2)\Rightarrow (3).$~
For any fixed cube $Q$, we have (see \cite{BMR} page 3333  or \cite{ZW2014} page 1383)
$$ M^{\sharp}(\chi_Q)(x) =\frac{1}{2},  ~~\mbox{for~all}~ x\in{Q}.
$$
Thus, for any $x\in{Q}$, we have
\begin{align*}
b(x)-2M^{\sharp}(b\chi_{Q})(x)&=2\bigg(\frac{1}{2}b(x)-M^{\sharp}(b\chi_{Q})(x)\bigg)\\
&=2\Big(b(x)M^{\sharp}(\chi_{Q})(x)-M^{\sharp}(b\chi_{Q})(x)\Big)\\
&=2[b,M^{\sharp}](\chi_{Q})(x).
\end{align*}

Therefore, by assertion (2) and noting that $1/q=1/p-\beta/n$, we obtain
\begin{align*}
&\frac{1}{\mu(Q)^{\beta/n}} \bigg(\frac{1}{\mu(Q)}
       \int_Q {|b(x)-2{M^{\sharp}}(b\chi_Q)(x)|^q \mu(x)^{1-q} dx}\bigg)^{1/q} \\
&=\frac{2}{\mu(Q)^{1/p}}\bigg(\int_{Q}\big|[b,M^{\sharp}](\chi_{Q})(x)\big|^{q}\mu(x)^{1-q}dx\bigg)^{1/q}\\
&=2\mu(Q)^{(\kappa-1)/p}\bigg(\frac{1}{\mu(Q)^{\kappa{q/p}}}
       \int_{Q}\big|[b,M^{\sharp}](\chi_{Q})(x)\big|^{q}\mu(x)^{1-q}dx\bigg)^{1/q}\\
&=2\mu(Q)^{(\kappa-1)/p} \big\|[b,M^{\sharp}](\chi_{Q})\big\|_{L^{q,\kappa{q/p}}(\mu^{1-q},\mu)}\\
&\leq C\mu(Q)^{(\kappa-1)/p} \|\chi_{Q}\|_{L^{p,\kappa}(\mu)}\\
&\leq C,
\end{align*}
where in the last step we have used Lemma \ref{lem.Q-Morrey}.
This concludes the proof of the implication that $(2)\Rightarrow (3)$,
 since the constant $C$ above is independent of $Q$.

The proof of Theorem \ref{thm.1.3} is finished.
\end{proof}

\end{document}